\documentclass[12pt,reqno]{amsart}
\usepackage{fullpage}
\allowdisplaybreaks
\usepackage{times}
\usepackage{colonequals}
\usepackage{amsmath,amssymb,amsthm,url}
\usepackage[utf8]{inputenc}
\usepackage[english]{babel}
\usepackage[alphabetic]{amsrefs}
\usepackage{bbm}
\usepackage{enumerate}
\usepackage{bm}
\usepackage{graphicx}
\usepackage{mathrsfs}
\usepackage[colorlinks=true, pdfstartview=FitH, linkcolor=blue, citecolor=blue, urlcolor=blue]{hyperref}

\newtheorem{thm}{Theorem}[section]
\newtheorem{lem}[thm]{Lemma}
\newtheorem{prop}[thm]{Proposition}
\newtheorem{cor}[thm]{Corollary}

\newtheorem{claim}[thm]{Claim}

\theoremstyle{definition}
\newtheorem{defn}[thm]{Definition}
\newtheorem{question}[thm]{Question}
\newtheorem{rem}[thm]{Remark}
\newtheorem{rems}[thm]{Remarks}
\newtheorem{ex}[thm]{Example}

\newcommand{\bP}{\mathbb{P}}

\newcommand{\C}{\mathbb{C}} 

\newcommand{\Z}{\mathbb{Z}}
\newcommand{\F}{\mathbb{F}}

\AtBeginDocument{%
	\def\MR#1{}
}

\begin{document}

\title{Plane curves giving rise to blocking sets over finite fields}

\author{Shamil Asgarli}
\address{Department of Mathematics and Computer Science \\ Santa Clara University \\ 500 El Camino Real \\ USA 95053}
\email{sasgarli@scu.edu}

\author{Dragos Ghioca}
\address{Department of Mathematics \\ University of British Columbia \\ 1984 Mathematics Road \\ Canada V6T 1Z2}
\email{dghioca@math.ubc.ca}

\author{Chi Hoi Yip}
\address{Department of Mathematics \\ University of British Columbia \\ 1984 Mathematics Road \\ Canada V6T 1Z2}
\email{kyleyip@math.ubc.ca}

\maketitle

\begin{abstract}
In recent years, many useful applications of the polynomial method have emerged in finite geometry. Indeed, algebraic curves, especially those defined by R\'edei-type polynomials, are powerful in studying blocking sets. In this paper, we reverse the engine and study when blocking sets can arise from rational points on plane curves over finite fields. We show that irreducible curves of low degree cannot provide blocking sets and prove more refined results for cubic and quartic curves. On the other hand, using tools from number theory, we construct smooth plane curves defined over $\mathbb{F}_p$ of degree at most $4p^{3/4}+1$ whose points form blocking sets.
\end{abstract}

\section{Introduction}\label{sect:intro}

Throughout the paper, $p$ denotes a prime, $q$ denotes a power of $p$, and $\F_q$ denotes the finite field with $q$ elements. 
A set of points $B\subseteq \bP^2(\F_q)$ is called a \emph{blocking set} if every $\F_q$-line $L$ intersects $B$. It is clear that taking $q+1$ points on a given $\mathbb{F}_q$-line forms a blocking set, since any two lines meet in the projective plane; this is known as a \emph{trivial} blocking set. A blocking set $B$ is said to be \emph{nontrivial} if it does not contain all the $\F_q$-points of any $\F_q$-line. 

It is known that the size of a nontrivial blocking set must satisfy $|B|\geq q+\sqrt{q}+1$ \cite{B70}. When $q=p$ is a prime, Blokhuis \cite{B94} proved a much stronger lower bound $|B|\geq \frac{3}{2}(p+1)$. The main tool used in Blokhuis' proof and in subsequent developments in this area have been R\'edei-type polynomials which are highly reducible algebraic curves. See the excellent surveys \cites{S97b, SS98} for more details and other applications of algebraic curves in finite geometry.

Let $C=\{F=0\}\subset \bP^2$ be an irreducible plane curve of degree $d\geq 2$ defined over a finite field $\F_q$. Let $C(\F_q)$ denote the set of $\F_q$-points on $C$, that is, $C(\F_q)$ consists of all $[x:y:z]\in \mathbb{P}^2(\F_q)$ such that $F(x,y,z)=0$. We assume that our curve is geometrically irreducible so that the curve is ``minimal" in the sense that it has no smaller proper component. Since smooth curves are important from the algebraic geometric point of view, we will sometimes further assume that $C$ is smooth. We are interested in the following problem.

\begin{question}\label{quest:version1}
When does there exist a line $L\subset \bP^2$ defined over $\F_q$ such that $C\cap L$ has no $\F_q$-points? 
\end{question}

One motivation for investigating Question~\ref{quest:version1} comes from algebraic geometry. Suppose $C$ is a curve that parametrizes other algebraic varieties; for instance, $C$ could be a curve in the parameter space of all degree $d$ hypersurfaces in $\mathbb{P}^n$. This means that a point on $C$ corresponds to a certain hypersurface of degree $d$. The points in $C(\F_q)$ would then correspond to those hypersurfaces of degree $d$ defined over $\mathbb{F}_q$. If we can find a line $L$ such that $C\cap L$ has no $\F_q$-points, then we have constructed a certain \emph{pencil} whose $\F_q$-members avoid $C$. In particular, one can use this idea to construct a pencil of hypersurfaces whose $\F_q$-members are smooth \cite{AG23}. 

Using the language of blocking sets, Question~\ref{quest:version1} is equivalent to determining when $C(\F_q)$ is \textbf{not} a blocking set. We remark that questions of a similar flavor have been studied in the past. In particular, Hirschfeld and Voloch \cite{HV88} asked when an arc is contained in an irreducible plane curve, and when an irreducible plane curve gives rise to a complete arc. One special motivation for constructing such curves lies in its application to coding theory \cites{B09}. Recall that a \emph{$(k,n)$-arc} in $\bP^2(\F_q)$ is a set of $k$ points where the maximum number of collinear points is $n$. There is an obvious relation between arcs and blocking sets: a complement of a $(k, n)$-arc in the plane is a (multiple) blocking set where each line meets the point set in at least $q+1-n$ points. The case of smooth conics was first studied by Segre \cite{S67}. Moreover, cubic curves, which give rise to ($k,3)$-arcs, were studied in \cites{HV88, G02, BMP17}. Some special algebraic constructions were discussed in \cites{GPTU02, B09, BMT14, GKT19, BM22,KPS23+}. Similar questions have been studied in the setting of caps in \cites{S67, ABGP14, ABPG15}. 

For simplicity, let us call $C$ a \emph{blocking curve} if $C(\F_q)$ is a blocking set. Furthermore, $C$ is \emph{nontrivially blocking} if $C(\F_q)$ is a nontrivial blocking set.  Constructing such curves is necessarily subtle, because most plane curves are not blocking \cite{AGY22b} from an arithmetic statistics perspective. 

Answering Question~\ref{quest:version1} in full generality seems difficult. Instead, we consider the following: for a given $d$ and $q$, does there exist a nontrivially blocking curve $C$ with degree $d$ defined over $\F_q$? The question is more interesting when $C$ is further assumed to be geometrically irreducible (or even smooth). Note that we are working over $\F_q$; however, it makes more sense to talk about geometric properties of curves over the algebraic closure $\overline{\F_q}$. When $d$ is large compared to $q$, one can study the proportion of smooth blocking curves \cite{AGY22b}. Thus, it makes sense to fix $q$, and ask for the minimum degree $d$ of a smooth blocking curve over $\F_q$.

Motivated by the past work on arcs arising from plane curves, we begin our study with the curves of low degree with respect to the cardinality of the field. Our first main result shows that an irreducible curve of low degree cannot be blocking.

\begin{thm}\label{thm:low-degree}
Let $C\subset \bP^2$ be an irreducible curve of degree $d\geq 4$. If $q \geq (d-1)^2(d-2)^4$, then $C$ is not blocking.
\end{thm}
We improve the bound $O(d^6)$ to $O(d^4)$ for the case $q=p^n$ with $n\leq 4$ in Theorem~\ref{thm-low-prime-power}.

When $d=2$, it is straightforward to see that an irreducible conic can never be a blocking set.  For cubic curves, we have a more refined result:

\begin{thm}\label{thm:cubic}
Let $C\subset \bP^2$ be an irreducible cubic curve. If $q \geq 5$, then $C$ is not blocking.
\end{thm}

Moreover, Theorem~\ref{thm:cubic} is sharp in the sense that there exist smooth cubic curves over $\F_q$ which are blocking when $q=2, 3, 4$. See Example~\ref{ex:cubic-small-fields}. We also establish a refined bound for smooth quartic curves.

\begin{thm}\label{thm:quartic} Suppose $C$ is a smooth plane curve of degree $4$ defined over a finite field $\mathbb{F}_q$. If $q\geq 19$, then $C$ is not blocking.
\end{thm}

Another aim of our paper is to construct explicit examples of smooth or irreducible blocking plane curves. When $q$ is a nontrivial prime power, meaning that $q=p^n$ with $n\geq 2$, then it is easy to find special curves which are smooth and blocking. For example, when $q$ is a square, we have the following well-known construction.

\begin{ex}\label{ex:hermitian}
Let $q$ be a square. Consider the Hermitian curve given by the equation
$$
\mathcal{H}: \ x^{\sqrt{q}+1} + y^{\sqrt{q}+1} + z^{\sqrt{q}+1} = 0. 
$$
It is known that every $\F_q$-line meets $\mathcal{H}$ in either $1$ or $\sqrt{q}+1$ points  \cite{S97b}*{Page 212}. In particular, $\mathcal{H}$ is a smooth curve of degree $d=\sqrt{q}+1$ such that $\mathcal{H}(\F_q)$ is a blocking set. 
\end{ex}

We provide a more general construction using Frobenius nonclassical curves in Section~\ref{sect:Frob-non-classical}. Such a construction relies on the subfield structure, which is not available in $\F_q$ when $q$ is a prime. When $q=p$ is a prime, it seems more difficult to find explicit examples of blocking curves; we find a family of examples in Theorem~\ref{thm:infinitely-many-3/4} and Theorem~\ref{thm:infinitely-many-3/4-smooth}.

It is natural to ask for the minimum degree of an irreducible curve passing through a specific blocking set. We analyze this question for the projective triangle in Section~\ref{sect:projective-triangle}. Recall that the \emph{projective triangle} (first mentioned in \cite{H79}) is the blocking set given by,
$$
\Delta = \{ [0:1:-s^2], [1:-s^2:0], [-s^2:0:1] \ | \ s\in\F_q  \}
$$
with cardinality $3(q+1)/2$. We construct a smooth curve with degree $\frac{q+3}{2}$ passing through $\Delta$ in Theorem~\ref{thm:projective-triangle} when $q \equiv 3 \pmod 4$. Indeed, such a curve would have degree at least $\frac{q+3}{2}$ by B\'ezout's theorem (see Remark~\ref{rem:bezout}), so our construction is optimal.

Note that, when $p$ is an odd prime, the projective triangle $\Delta$ is of particular interest, since it serves as an example of a nontrivial blocking set of size $\frac{3}{2}(p+1)$ over $\F_p$, which is the smallest possible size by Blokhuis \cite{B94}. However, this does not imply that the smallest degree irreducible blocking curve $C$ must necessarily pass through the projective triangle (the smallest blocking set) or its image under a projective transformation. Indeed, we find geometrically irreducible blocking curves with smaller degree $d=\frac{p-1}{r}+1$ in Theorem~\ref{thm:irreducible-curves} for a fixed $r$, provided that $p \equiv 1 \pmod r$ and $p>r^4$. Using tools from analytic number theory, we can prove something even stronger.

\begin{thm}\label{thm:infinitely-many-3/4} There are infinitely many primes $p$ such that for each $d \geq 4p^{3/4}+1$, there is a geometrically irreducible nontrivially blocking curve over $\F_p$ with degree $d$.
\end{thm}

In view of Theorem~\ref{thm:infinitely-many-3/4}, one can ask for the \emph{smallest} degree of an irreducible blocking curve over $\F_p$ for $p$ prime. According to Theorem~\ref{thm-low-prime-power}, the optimal degree for a blocking curve satisfies $d\geq C_0 p^{1/4}$ for some constant $C_0$, while Theorem~\ref{thm:infinitely-many-3/4} tells us that the optimal degree must satisfy $d\leq C_1 p^{3/4}$ for some constant $C_1$. Let us briefly explain why the optimal exponent of degree is likely to be near $1/2$. We expect that for any $\varepsilon > 0$, there are many blocking sets of size at most $\lambda_{\varepsilon} p^{1+\varepsilon}$ where $\lambda_{\varepsilon}$ is a constant. Since the vector space $V$ of degree $d$ homogeneous polynomials defining plane curves has dimension $\binom{d+2}{2}$, we obtain many blocking curves provided that $\binom{d+2}{2} > \lambda_{\varepsilon} p^{1+\varepsilon}$. Indeed, passing through any specific point imposes at most one linear condition on $V$. While this furnishes numerous blocking curves of degree $d \leq C_{\varepsilon} p^{1/2+\varepsilon}$, this is only a heuristic because we cannot demonstrate irreducibility of any such curve in this abstract setting.

Constructing smooth blocking curves appears to be much more difficult. One difficulty in applying the heuristic above is the following: we want the smooth curve to pass through a blocking set with relatively large size, while we expect that the number of $\F_p$-points of most smooth curves is close to $p+1$. Nonetheless, we prove a version of Theorem~\ref{thm:infinitely-many-3/4} for smooth blocking curves:

\begin{thm}\label{thm:infinitely-many-3/4-smooth}
Let $0<\theta \leq 1/4$. Let $A$ be a fixed positive number, with $A \leq 1/2$ if $\theta=1/4$. There are infinitely many primes $p$ such that for \emph{some} $d \in [p^{1-\theta}/2A+1, 2p^{1-\theta}/A+1]$, there is a smooth nontrivially blocking curve over $\F_p$ with degree $d$.
\end{thm}

The proof of Theorem~\ref{thm:infinitely-many-3/4-smooth} relies on a general construction of a smooth blocking curve (Theorem~\ref{thm:smoothcurve}). In particular, we can construct smooth blocking curves of degree $\frac{q+1}{2}$. Note that the existence of degree $d$ smooth blocking curve does not necessarily imply the existence of degree $d+1$ smooth blocking curve. Therefore, the next result would not follow from knowing the existence of smooth blocking curves of degree $d=\frac{q+1}{2}$. 

\begin{thm}\label{thm:construction-smooth-(q+3)/2}
Suppose $q\geq 5$ with $p=\operatorname{char}(\F_q)>3$. There exists a smooth blocking plane curve over $\F_q$ with degree $d=\frac{q+3}{2}$.
\end{thm}

Similarly, the existence of degree $d$ smooth blocking curve does not necessarily imply the existence of degree $d-1$ smooth blocking curve. However, we are able to exhibit a smooth blocking set of degree $\frac{q-1}{2}$ whenever $q\equiv 3\pmod{4}$.

\begin{thm}\label{thm:construction-smooth-(q-1)/2}
Suppose $q\geq 11$ with $q\equiv 3\pmod{4}$ and $p=\operatorname{char}(\F_q)>3$. There exists a smooth blocking plane curve over $\F_q$ with degree $d=\frac{q-1}{2}$.
\end{thm}

When $q\equiv 1\pmod{4}$, we expect that there should be examples with degree $\frac{q-1}{2}$ for $q\geq 13$. See Example~\ref{ex:degree-m}.

\subsection*{Structure of the paper.} In Section~\ref{sect:prelims}, we present some preliminary definitions, and discuss useful tools from number theory, algebra, and incidence geometry. In Section~\ref{sect:low-degree}, we employ point-line incidence geometry to prove Theorem~\ref{thm:low-degree}, Theorem~\ref{thm:cubic} and Theorem~\ref{thm:quartic}. We turn our attention to special constructions of blocking curves using Frobenius nonclassical plane curves in Section~\ref{sect:Frob-non-classical}. In Section~\ref{sect:projective-triangle} we construct smooth curves passing through the projective triangle. Finally, we prove Theorem~\ref{thm:infinitely-many-3/4} and Theorem~\ref{thm:infinitely-many-3/4-smooth} in Section~\ref{sect:constructions-I}, and Theorem~\ref{thm:construction-smooth-(q+3)/2} and Theorem~\ref{thm:construction-smooth-(q-1)/2} in Section~\ref{sect:constructions-II}.

\section{Preliminaries and Lemmata}\label{sect:prelims}

This section is not meant to be read in isolation, and a reader may skip to Section~\ref{sect:low-degree} and refer back to this section when necessary.

\subsection{Basic definitions.} The primary geometric object of our study is a plane curve defined over a finite field. Recall that a plane curve $C\subset\mathbb{P}^2$ is defined by a homogeneous polynomial $F\in\mathbb{F}_q[x,y,z]$. We say that $C$ is \emph{irreducible} if $F$ is an irreducible polynomial in $\mathbb{F}_q[x,y,z]$. Moreover, $C$ is \emph{geometrically irreducible} if $F$  remains irreducible in the larger ring $\overline{\F_q}[x,y,z]$. A plane curve $C$ is \emph{smooth} if for every point $P\in C$, at least one of the partial derivatives $F_x, F_y$ or $F_z$ does not vanish at $P$. Note that a smooth plane curve is geometrically irreducible.

Recall that $\Phi\colon \bP^2\to \bP^2$ is the \emph{$q$-th power Frobenius map} defined by $\Phi[x:y:z]=[x^q:y^q:z^q]$. This definition will be especially useful in Section~\ref{sect:Frob-non-classical} when we discuss Frobenius nonclassical curves.

\subsection{Character sums}
Recall that a character $\chi$ of the multiplicative group $\F_q^*$ of $\F_q$ is called a {\em multiplicative character} of $\F_q$. For a multiplicative character $\chi$, its order $r$ is the smallest positive integer such that $\chi^r=\chi_0$, where $\chi_0$ is the trivial multiplicative character of $\F_q$. The following lemma is a classical application of Weil's bound for complete character sums; see for example \cite{LN97}*{Exercise 5.66}. 

\begin{lem}\label{lem:countingsolns}
Let $r\geq 2$ be a positive integer. Let $q\equiv 1\pmod{r}$ be a prime power. Let $\chi$ be a multiplicative character of $\F_q$ with order $r$. Let $a_1,a_2,\ldots, a_k$ be $k$ distinct elements of $\F_q$, and let $\epsilon_1,\ldots, \epsilon_k \in \C$ be $r$-th roots of unity. Let $N$ denote the number of $x \in \F_q$ such that $\chi(x+a_i)=\epsilon_i$ for $1 \leq i \leq k$. Then 
$$
N \geq \frac{q}{r^k}-\bigg(k-1-\frac{k}{r}+\frac{1}{r^k}\bigg)\sqrt{q}-\frac{k}{r}.
$$
\end{lem}

We apply Lemma~\ref{lem:countingsolns} to deduce two useful corollaries below. They will be used to show that the curves in Section~\ref{sect:constructions-I} and Section~\ref{sect:constructions-II} are blocking.

\begin{cor}\label{cor:by+cz}
Let $r\geq 2$ be a positive integer. Let $q\equiv 1\pmod{r}$ be a prime power such that $q>r^4$. Let $\chi$ be a multiplicative character of $\F_q$ with order $r$, and let $\omega_1$, $\omega_2 \in \C$ be two $r$-th roots of unity. Then for any $b,c \in \F_q^*$, there are $y,z \in \F_q$ such that $\chi(y)=\omega_1, \chi(z)=\omega_2$ and $by+cz=-1$. 
\end{cor}

\begin{proof}
Note that $by+cz=-1$ is equivalent to $z=-\frac{b}{c}(y+\frac{1}{b})$. Thus, it suffices to find $y \in \F_q$ such that $\chi(y)=\omega_1$ and $\chi(y+\frac{1}{b})=w_2 \chi(-\frac{c}{b})$. Using Lemma~\ref{lem:countingsolns} with $k=2$, the number of such $y$ is at least 
$$
\frac{q}{r^2}-\bigg(1-\frac{2}{r}+\frac{1}{r^2}\bigg)\sqrt{q}-\frac{2}{r}=\bigg(\frac{q}{r^2}-\sqrt{q}\bigg)+\bigg(\frac{2}{r}-\frac{1}{r^2}\bigg)\sqrt{q}-\frac{2}{r}> \frac{\sqrt{q}-2}{r}>0
$$
since $q>r^4$. Thus, such an element $y$ exists.
\end{proof}

\begin{cor}\label{cor:47}
Let $q$ be an odd prime power such that $q \geq 47$, and $\chi$ be the quadratic character of $\F_q$. Then for any nonzero $\alpha, \beta \in \F_q$ such that $\alpha \neq \beta$, there is $x \in \F_q$, such that $\chi(x), \chi(x+\alpha), \chi(x+\beta)$ have the prescribed values ($1$ or $-1$).
\end{cor}

\begin{proof}
For $q \geq 47$, we have
$$
\frac{q}{8}> \frac{5}{8}\sqrt{q}+\frac{3}{2}.
$$
The result follows immediately from Lemma~\ref{lem:countingsolns} with $k=3$ and $r=2$.
\end{proof}

\subsection{Irreducibility criterion}

In Section~\ref{sect:constructions-I} we will construct a family of geometrically irreducible blocking curves. The following lemma provides a useful criterion to check the (absolute) irreducibility of a polynomial.

\begin{lem}\label{lem:eisenstein-criterion}
Let $K$ be a field and let $f, g\in K[y, z]$ homogeneous polynomials such that $\gcd(f, g)=1$. Assume that either $f$ or $g$ has a non-repeated irreducible factor. Then the polynomial
$$
F(x,y,z) = f(y,z) x^r + g(y,z) 
$$
is irreducible for each $r \geq 1$.
\end{lem}

\begin{proof} First, suppose $h(y,z)$ is an irreducible polynomial such that $h\mid g$ but $h^2\nmid g$. Since $\gcd(f, g)=1$, we know that $h\nmid f$. By Eisenstein's criterion with the underlying ring $K[y,z]$ (see for example \cite{G01}*{Page 502}), the polynomial $f(y,z) x^r + g(y,z)$ is irreducible in the ring $(K[y,z])[x]$. Since $\gcd(f, g)=1$, the polynomial is also irreducible in $K[x,y,z]$. 

Next, suppose that $h(y,z)$ is an irreducible polynomial such that $h\mid f$ but $h^2\nmid f$. By applying the same argument in the first paragraph, the polynomial $g(y,z) x^r+f(y,z)$ is irreducible. It follows that $f(y,z) x^r + g(y,z)$ is also irreducible. This is because the two polynomials are related by the transformation $x\mapsto 1/x$ and multiplication by $x^{r}$.
\end{proof}

\subsection{Divisors of  \texorpdfstring{$p-1$}{p-1}} 
In this subsection, we show that for any $\theta<1/2$, there are infinitely many primes $p$ such that $p-1$ has a divisor that is close to $p^{\theta}$. We first recall some standard notations. For any real number $x$, let $\pi(x)$ be the number of primes up to $x$. For positive integers $r$ and $a$ such that $\gcd(r, a)=1$, we use $\pi(x;r, a)$ to denote the number of primes up to $x$ which are in the arithmetic progression $a+r\Z$. The prime number theorem for arithmetic progressions states that $\pi(x;r, a)$ is very close to $\pi(x)/\phi(r)$ if $r$ is fixed and $x$ is large. For our purposes, $r$ is close to $x^\theta$, so we need deeper tools from analytic number theory. 

The following explicit version of the Brun-Titchmarsh inequality is due to Montgomery and Vaughan \cite{MV73}. 

\begin{thm}[Brun-Titchmarsh inequality]
If $x>r$, then
$$
\pi(x;r,a) \leq \frac{2x}{\phi(r)\log \frac{x}{r}}.
$$
\end{thm}

We also need the following version of the celebrated Bombieri–Vinogradov theorem \cite{B65}.

\begin{thm} [Bombieri–Vinogradov theorem]
Let $\theta<1/2$. There is a constant $C$, such that
\begin{equation}\label{BV}
\sum_{r \leq x^{\theta}} \max_{\gcd(a,r)=1} \bigg|\pi(x;r,a)-\frac{\pi(x)}{\phi(r)}\bigg|\leq \frac{Cx}{(\log x)^2}.    
\end{equation}
\end{thm}

\begin{cor}\label{cor:divisor}
Let $0<\theta<1/2$ and $0<A$ be fixed constants. There are infinitely many primes $p$ such that $p-1$ has a divisor in $[Ap^{\theta}/2,2Ap^{\theta}]$.
\end{cor}

\begin{proof}
We prove the statement for the case $A=1$; the proof for the general case is similar (but a bit messier). 

Let $0<\alpha<1/4$ such that $\alpha^{-\theta}<2$. 
Let $x$ be sufficiently large so that $\alpha x^{1-\theta}>x^{1/2}$. By Bombieri–Vinogradov theorem, the inequality~\eqref{BV} holds. In particular, we can find some $r_0 \in [x^\theta/2, x^\theta]$ such that
$$
\bigg|\pi(x;r_0,1)-\frac{\pi(x)}{\phi(r_0)}\bigg|\leq \frac{2Cx}{x^{\theta}(\log x)^2}<\frac{2Cx}{\phi(r_0)(\log x)^2}.
$$
On the other hand, Brun-Titchmarsh inequality implies that
$$
\pi(\alpha x;r_0,1) \leq \frac{2\alpha x}{\phi(r_0)\log \frac{\alpha x}{r_0}}\leq \frac{2\alpha x}{\phi(r_0)\log (\alpha x^{1-\theta})}\leq \frac{2\alpha x}{\phi(r_0)\log ( x^{1/2})}=\frac{4\alpha x}{\phi(r_0)\log x}.
$$
Combining the above estimates, we have
$$
\pi(x;r_0,1)-\pi(\alpha x; r_0,1) \geq \frac{\pi(x)}{\phi(r_0)}-\frac{2Cx}{\phi(r_0)(\log x)^2}-\frac{4\alpha x}{\phi(r_0)\log x}.
$$
Note that $4\alpha<1$, so the prime number theorem implies that $\pi(x;r_0,1)-\pi(\alpha x;r_0,1)>0$ holds for sufficiently large $x$. In particular, there is a prime $p \in (\alpha x, x]$ such that $p \equiv 1 \pmod {r_0}$ with $r_0 \in [x^\theta/2, x^\theta]$, which implies that $r_0 \in [p^\theta/2, (p/\alpha)^\theta] \subset [p^\theta/2, 2p^\theta]$.
\end{proof}

\begin{rem}
Let $P(a,r)$ be the least prime in an arithmetic progression $a \pmod r$, where $a$ and $r$ are coprime positive integers. Linnik \cites{L44a,L44b} proved that there exist effectively computable constants $C$ and $L$ such that $P(a,r) \le Cr^L$. The constant $L$ is known as {\em Linnik's constant}. The best-known upper bound for $L$ is $5$, due to Xylouris~\cite{X11}. Using Linnik's theorem with $r=2^n$ and $L=5$, one can give a simple proof for Corollary~\ref{cor:divisor} when $\theta<1/5$. However, in our application, we need the statement to be true for $\theta\leq 1/4$. We also remark that Corollary~\ref{cor:divisor} is trivial if one replaces primes with prime powers.
\end{rem}

\subsection{Point-line incidences} The following lemma collects the three useful identities that will be repeatedly used in Section~\ref{sect:low-degree}.

\begin{lem}\label{lemma:three-key-identites}
Let $C$ be an irreducible blocking plane curve of degree $d$ defined over a finite field $\F_q$. Suppose that $t_i$ denotes the number of $\F_q$-lines intersecting $C(\F_q)$ in exactly $i$ points. Let $N=|C(\F_q)|$. Then we have $t_0=0$ and $t_i=0$ when $i>d$. Moreover,
\begin{align*} 
& \sum_{i=1}^{d} t_i=q^2+q+1, \\
& \sum_{i=1}^{d} it_i=(q+1)N, \\
& \sum_{i=2}^d t_i \binom{i}{2}=\binom{N}{2}.
\end{align*} 
\end{lem}

\begin{proof}
The proof relies on a standard counting of point-line incidences in two different ways. Note that the sums run up to $i=d$, because $t_i=0$ for $i>d$ by B\'ezout's theorem which is applicable as $C$ is irreducible. For complete details, see for example the proof of \cite{AG23}*{Proposition 2.1}.
\end{proof}

\begin{rem}
The sequence $\{t_i\}$, known as the \emph{intersection distribution}, makes sense for any point set $S$ in $\mathbb{P}^2(\F_q)$. Such distribution is closely related to many problems in finite geometry \cite{LP20}*{Remark 1.4}.
In particular, Li and Pott \cite{LP20} studied the sequence $\{t_i\}$ when $S$ is the graph of a polynomial $f\in \mathbb{F}_q[x]$, and found connections to permutation polynomials \cite{AGW11}.
\end{rem}

As an immediate corollary of B\'ezout's theorem (or Lemma~\ref{lemma:three-key-identites}), we have a criterion for showing that a blocking curve is nontrivially blocking.

\begin{cor}\label{cor:d<q+1-nontrivial} If $C\subset \bP^2$ is an irreducible blocking plane curve of degree $d<q+1$, then $C$ is nontrivially blocking.
\end{cor}

As another corollary of Lemma~\ref{lemma:three-key-identites}, we obtain the following inequality which relates the degree and the number of rational points of a blocking curve, and the cardinality of the ground field. 

\begin{cor}\label{cor:key-inequality}
Let $C$ be an irreducible blocking plane curve of degree $d$ over $\F_q$. Using the notation in Lemma~\ref{lemma:three-key-identites}, the following inequality holds:
$$
N(d(q+1)+1-N)\geq d(q^2+q+1).
$$
\end{cor}

\begin{proof}
By Lemma~\ref{lemma:three-key-identites}, we have
$$
\sum_{i=1}^{d} t_i=q^2+q+1, \quad \sum_{i=1}^{d} it_i=(q+1)N, \quad \sum_{i=1}^d t_i \binom{i}{2}=\binom{N}{2}
$$
where the third sum starts with $i=1$ for convenience. Now, the last two equations imply that
$$
\sum_{i=1}^d i^2t_i=2\sum_{i=1}^d t_i \binom{i}{2}+\sum_{i=1}^{d} it_i=N(N-1)+(q+1)N=N(N+q).
$$
Thus,
\begin{equation*}\label{it_i}
\sum_{i=1}^{d} (i-1)t_i=(q+1)N-(q^2+q+1),
\end{equation*}
which implies that
$$
\sum_{i=1}^{d} (i-1)^2t_i\leq (d-1) \sum_{i=1}^{d} (i-1)t_i= (d-1)(q+1)N-(d-1)(q^2+q+1).
$$
On the other hand,
\begin{equation*}
\sum_{i=1}^{d} (i-1)^2t_i=\sum_{i=1}^{d} i^2t_i-2\sum_{i=1}^{d} it_i+\sum_{i=1}^{d} t_i=N(N+q)-2(q+1)N+q^2+q+1.
\end{equation*}
Combining the two estimates above, we get
$$
(d-1)(q+1)N-(d-1)(q^2+q+1) \geq N(N+q)-2(q+1)N+q^2+q+1=N(N-q-2)+q^2+q+1.
$$
Simplifying yields
\begin{equation*}\label{ineq-key-thm-lower-bound}
N(d(q+1)+1-N)\geq d(q^2+q+1),    
\end{equation*}
as desired.
\end{proof}

\section{Low degree plane curves are not blocking}\label{sect:low-degree}

In this section, we prove Theorem~\ref{thm:low-degree} to establish that curves of low degree can never be blocking, and present an improved result in Theorem~\ref{thm-low-prime-power} for the case $q=p^n$ when $n\leq 4$. We also prove Theorem~\ref{thm:cubic} and Theorem~\ref{thm:quartic}, which provide refined results for cubic and quartic curves, respectively.

\subsection{Proof of Theorem~\ref{thm:low-degree}}

Before we proceed with the proof of Theorem~\ref{thm:low-degree}, we present a quick lemma that allows us to reduce to the case of geometrically irreducible curves.

\begin{lem}\label{lem:geometric-irred}
Let $C$ be an irreducible plane curve of degree $d$ defined over $\F_q$. Suppose that $C$ is not geometrically irreducible. If $q\geq d^2/4$, then $C$ is not blocking.
\end{lem}

\begin{proof}
Under the given hypothesis, it is well-known that $|C(\F_q)|\leq d^2/4$  \cite{AG23}*{Remark 2.2}. Thus, the total number of $\F_q$-lines passing through some point of $C(\F_q)$ is at most $$\frac{d^2}{4}\cdot (q+1)\leq q(q+1) < q^2+q+1$$ and hence there is some $\F_q$-line $L$ which does not meet $C(\F_q)$. Thus, $C$ is not blocking. 
\end{proof}

Next, we give a proof of Theorem~\ref{thm:low-degree}. We remark that Theorem~\ref{thm:low-degree} improves \cite{AG23}*{Proposition 2.1} by the multiplicative factor given by $(\frac{1+\sqrt{2}}{2})^2 \approx 1.457$.

\begin{proof}[Proof of Theorem~\ref{thm:low-degree}]
Suppose, to the contrary, that $C$ is blocking. In view of Lemma~\ref{lem:geometric-irred}, we may assume that $C$ is geometrically irreducible. 

Let $N$ denote the number of $\F_q$-points of $C$. Motivated by Corollary~\ref{cor:key-inequality}, we are led to consider the function $f(x)\colonequals x(d(q+1)+1-x)$.  Note that $f(x)$ increases as a function of $x$ whenever $0\leq x<\frac{d(q+1)+1}{2}$. 

Let us first show that $N$ must be in the interval $\left(0, \frac{d(q+1)+1}{2}\right)$. By the Hasse-Weil bound for geometrically irreducible curves \cite{AP96}*{Corollary 2.5}:
$$N \leq q+1+(d-1)(d-2)\sqrt{q}.$$
It suffices to establish $q+1+(d-1)(d-2)\sqrt{q}<\frac{d(q+1)+1}{2}$. Since $d \geq 4$ and $q\geq (d-1)^2(d-2)^4$, it follows that $q\geq 16(d-1)^2$. Therefore, 
\begin{align*}
\frac{d(q+1)+1}{2}-q-1-(d-1)(d-2)\sqrt{q}
&>\frac{(d-2)(q+1)}{2}-(d-1)(d-2)\sqrt{q}\\
&>(d-2)\left(\frac{q}{2}-(d-1)\sqrt{q}\right)\geq 0
\end{align*}
as desired. This shows $N$ is in the interval where $f(x)$ is increasing. As a result,
\begin{align*}
& N(d(q+1)+1-N) = f(N) \leq f(q+1+(d-1)(d-2)\sqrt{q})= \\
&= (d-1)(q+1)^2+q+1+((d-2)(q+1)+1)(d-1)(d-2)\sqrt{q}-(d-1)^2(d-2)^2q.
\end{align*}
Thus, Corollary~\ref{cor:key-inequality} implies that
$$
q^2 \leq 2(d-1)q+(d-1)(d-2)^2q\sqrt{q}+(d-1)^2(d-2)\sqrt{q}-(d-1)^2(d-2)^2q.
$$
Since $q \geq 4(d-1)^2$ and $d \geq 4$, we have
$$
(d-1)^2(d-2)^2q \geq 2(d-1)^3(d-2)^2\sqrt{q}, \quad  (d-1)^2(d-2)^2q \geq 4(d-1)q,
$$
thus
$$
(d-1)^2(d-2)^2q\geq  2(d-1)q+(d-1)(d-2)^2\sqrt{q}.
$$
We conclude that 
$$
q^2 \leq (d-1)(d-2)^2q\sqrt{q},
$$
that is, $q\leq (d-1)^2(d-2)^4$. Since $q$ is a prime power and $(d-1)^2(d-2)^2$ has at least two distinct prime factors for $d\geq 4$, we deduce that $q<(d-1)^2(d-2)^4$. This is a contradiction, and the proof is complete.
\end{proof}

Before discussing our next result, we mention a fundamental result on blocking sets. A blocking set of size less than $3(q+1)/2$ is known as a {\em small blocking set}. Sz{\H{o}}nyi \cite{S97}*{Corollary 4.8} proved that a small blocking set must have a special incidence structure.

\begin{thm}[\cite{S97}] \label{1modp}
If $B \subset \bP^2(\F_q)$ is a nontrivial blocking set of size less than $3(q+1)/2$, then
each line intersects $B$ in $1$ modulo $p$ points. 
\end{thm}

When $q=p^n$, where $n \leq 4$, the $O(d^6)$ bound in Theorem~\ref{thm:low-degree} can be further improved to $O(d^4)$.

\begin{thm}\label{thm-low-prime-power}
Let $C$ be an irreducible plane curve of degree $d\geq 4$ defined over $\F_q$, where $q=p^n$ such that $n \in \{1,2,3,4\}$. If $q\geq 4(d-1)^2(d-2)^2$, then $C$ is not blocking.
\end{thm}

\begin{proof}

Suppose to the contrary that $C$ is blocking. In view of Lemma~\ref{lem:geometric-irred}, we may assume that $C$ is geometrically irreducible. Using the Hasse-Weil bound and the hypothesis $\frac{q}{4}\geq (d-1)^2(d-2)^2$, we obtain:
\begin{align*}
    |C(\F_q)|\leq  q+1 + (d-1)(d-2)\sqrt{q} 
    \leq q+1 + \frac{\sqrt{q}}{2} \sqrt{q} 
    = \frac{3q}{2}+1 < \frac{3(q+1)}{2}.
\end{align*}
Since $C$ is irreducible and $d < q+1$, it follows that $C$ is nontrivially blocking by Corollary~\ref{cor:d<q+1-nontrivial} (that is, $C(\F_q)$ does not contain all the $\F_q$-points of an $\F_q$-line since that would make $C(\F_q)$ a blocking set in a trivial manner). By Theorem~\ref{1modp}, each line intersects $C(\F_q)$ in $1$ modulo $p$ points. Note that $p^4 \geq q \geq 4(d-1)^2(d-2)^2$, so $p^2\geq 2(d-1)(d-2)$, which implies that $p \geq d$ provided that $d \geq 4$. This forces $t_1=q^2+q+1$ and $t_i=0$ for $i>1$ (since $p \geq d$) by Lemma~\ref{lemma:three-key-identites}. Applying Lemma~\ref{lemma:three-key-identites} again, we obtain the equation, $$(q+1)N=\sum_{i=1}^d it_i=q^2+q+1$$ which is impossible since $(q+1) \nmid (q^2+q+1)$. \end{proof}

\begin{rem}
When $q=p$ is a prime, one can obtain a simpler proof of Theorem~\ref{thm-low-prime-power} by comparing the lower bound given by $\frac{3}{2}(p+1)$ in Blokhuis' theorem \cite{B94} with the Hasse-Weil bound.
\end{rem}

\subsection{Cubic plane curves} 

We will next show that a cubic plane curve is almost never blocking. As a reference, cubic plane curves are discussed in \cite{H79}*{Chapter 11}.

\begin{proof}[Proof of Theorem~\ref{thm:cubic}]
Assume, to the contrary, that $C(\F_q)$ is a blocking set. Let $N=\# C(\F_q)$. As before, let $t_i$ denote the number of $\F_q$-lines intersecting $C(\F_q)$ in $i$ points. By Lemma~\ref{lemma:three-key-identites},
\begin{align*} 
& t_1 + t_2 + t_3 = q^2 + q+1, \\
& t_1 + 2t_2 + 3 t_3 = N(q+1), \\
& 2 t_2 + 6 t_3 = N(N-1). 
\end{align*} 
Subtracting the first equation from the second, we get $t_2+2t_3=N(q+1)-(q^2+q+1)$. Combining this equation with the third displayed equation, we get:
\begin{align*}
t_2 &= 3(t_2+2t_3) - (2t_2+6t_3) \\
&= 3N(q+1)-3(q^2+q+1) - N(N-1).
\end{align*}
Since $t_2\geq 0$, we obtain:
$$
3N(q+1)-3(q^2+q+1) - N(N-1)\geq 0
$$
which is a quadratic inequality in $N$:
\begin{align}\label{eq:d=3:quad-inequality}
N^2 - (3q+4)N + 3(q^2+q+1) \leq 0.
\end{align}
The discriminant is given by:
$$
\Delta = (3q+4)^2 - 12(q^2+q+1) = -(3q^2-12q-4) < 0 $$
for $q\geq 5$. Indeed, $3q^2-12q-4 = 3(q-5)^2 + 18q-79 > 0$ for $q\geq 5$. Since $\Delta<0$, the quadratic function $f(N)=N^2 - (3q+4)N + 3(q^2+q+1)$ has no real roots, and therefore should always be positive. This contradicts the inequality \eqref{eq:d=3:quad-inequality}, and completes the proof. \end{proof}

\begin{ex}\label{ex:cubic-small-fields} Consider the following three plane cubic curves:
\begin{itemize}
    \item $C_2: x^3+y^3+z^3=0 $ defined over $\F_2$;
    \item $C_3: y^2 z - x^3 - x^2 z - x z^2=0 $ defined over $\F_3$;
    \item $C_4: x^3+y^3+z^3=0$ defined over $\F_4$.
\end{itemize}
One can show that $C_i$ is a smooth blocking curve for each $i=2, 3, 4$. Note $C_4$ is an example of Hermitian curve mentioned earlier in Example~\ref{ex:hermitian}. Thus, the condition $q\geq 5$ required in Theorem~\ref{thm:cubic} is both necessary and sharp.
\end{ex}

\begin{rem}
The irreducibility hypothesis on the cubic curve $C$ is necessary. Indeed, a reducible cubic curve $C$ over $\F_q$ must have a component of degree $1$, that is, an $\F_q$-line. But then $C(\F_q)$ would be a trivial blocking set.
\end{rem}

\begin{rem}
The same proof works for $(k, 3)$-arcs. More precisely, if $\mathcal{A}$ is a $(k, 3)$-arc in $\bP^2(\F_q)$ with $q\geq 5$, then $\mathcal{A}$ is not a blocking set. \end{rem}

\subsection{Quartic plane curves} 

Obtaining a more refined result for quartic plane curves is more complicated, and our proof below crucially relies on the smoothness hypothesis. 

\begin{proof}[Proof of Theorem~\ref{thm:quartic}]
Assume, to the contrary, that $C(\F_q)$ is a blocking set. By Lemma~\ref{lemma:three-key-identites}, 
\begin{align*} 
& t_1 + t_2 + t_3 + t_4 = q^2 + q+1, \\
& t_1 + 2t_2 + 3 t_3 + 4t_4 = N(q+1), \\
& 2 t_2 + 6 t_3 + 12 t_4 = N(N-1). 
\end{align*} 
We subtract the first equation from the second to get:
\begin{equation}
\label{eq:1001}
t_2 + 2t_3 + 3t_4 = N(q+1) - (q^2+q+1).
\end{equation}
Now, 
\begin{equation}
\label{eq:1002}
t_2 + 2t_3 + 3t_4 = \frac{t_2+t_3}{2}+ \frac{2 t_2 + 6 t_3 + 12 t_4}{4}\geq \frac{t_2}{2} + \frac{N(N-1)}{4}.
\end{equation}
Let $P\in C(\F_{q^2})\setminus C(\F_q)$. Consider the $\F_q$-line $L_P$ joining $P$ and its Frobenius image $\Phi(P)$. Note that $L_P=L_{\Phi(P)}$. We claim that the number of such lines is exactly 
\begin{equation}
\label{eq:1003}
\frac{\# C(\F_{q^2})- \#C(\F_{q})}{2}. 
\end{equation}
In order to prove the formula~\eqref{eq:1003}, it suffices to show that $L_P=L_{P'}$ if and only if $\{P, \Phi(P)\}=\{P', \Phi(P')\}$. Assuming $\{P, \Phi(P)\}\neq \{P', \Phi(P')\}$, the condition $L_P=L_{P'}$ would imply that then $L_P \cap C$ has $4$ points, namely $P, \Phi(P), P', \Phi(P')$, implying that $L_P \cap C(\F_q)=\emptyset$, contradicting the assumption that $C$ is blocking. This completes the proof of our claim that the number of distinct lines $L_P$ (joining $P$ and its Frobenius image $\Phi(P)$) is given by the formula~\eqref{eq:1003}.

Given $P\in C(\F_{q^2})\setminus C(\F_q)$, we know that $L_P\cap C$ has $4$ points over $\overline{\F_q}$ counted with multiplicity. Moreover, $2$ of those $4$ points are already accounted by $P$ and $\Phi(P)$, neither of which is an $\F_q$-point. Thus, either $L_P$ contributes to $t_2$ or $L$ is a tangent line to $C$ at some $\F_q$-point, that is, $L$ meets $C$ at some $\F_q$-point with multiplicity exactly $2$. Since $C$ is smooth, the number of tangent lines is at most $N=\#C(F_q)$. Therefore,
\begin{equation}
\label{eq:1004}
t_2 \geq \frac{\# C(\F_{q^2})- \#C(\F_{q})}{2} - N.
\end{equation}
Using the Hasse-Weil bound, $\#C(\F_{q^2})\geq q^2+1-6q$. Thus, equation~\eqref{eq:1004} yields that
\begin{equation}
\label{eq:1005}
t_2 \geq \frac{q^2+1-6q-N}{2} - N.
\end{equation}
Substituting the lower bound for $t_2$ from equation~\eqref{eq:1005} into the equation~\eqref{eq:1002}, we obtain 
\begin{equation}
\label{eq:1006}
t_2 + 2t_3 + 3t_4 \geq \frac{q^2+1-6q-N}{4} - \frac{N}{2} + \frac{N(N-1)}{4}.
\end{equation}
Therefore, equations~\eqref{eq:1001}~and~\eqref{eq:1006} yield that 
\begin{equation}
\label{eq:1007}
N(q+1) - (q^2+q+1) \geq \frac{q^2+1-6q-N}{4} - \frac{N}{2} + \frac{N(N-1)}{4}.
\end{equation}
Rearranging inequality~\eqref{eq:1007} into a quadratic equation in $N$, we obtain:
\begin{equation}
\label{eq:1008}
\frac{1}{4}N^2 - \left(2 + q\right) N + \frac{5}{4} q^2 - \frac{1}{2} q + \frac{5}{4}  \leq 0.
\end{equation}
The discriminant of the quadratic from~\eqref{eq:1008} is given by:
$$
\Delta = \left(2 + q\right)^2 - \left(\frac{5}{4} q^2 - \frac{1}{2} q + \frac{5}{4}\right) = -\frac{1}{4}\left(q^2-18q-11 \right).
$$
Note that $\Delta<0$ for $q\geq 19$, which contradicts the earlier inequality $\frac{1}{4}N^2 - \left(2 + q\right) N + \frac{5}{4} q^2 - \frac{1}{2} q + \frac{5}{4}  \leq 0$. Thus, the desired conclusion from Theorem~\ref{thm:quartic} is proved for $q\geq 19$. \end{proof}

\begin{rems}\label{remark:quartic-curves}
The Hermitian curve from Example~\ref{ex:hermitian} gives an example of a smooth quartic blocking curve when $q=9$. We do not know if there exist smooth or irreducible blocking plane curves of degree $4$ over $\F_q$ when $q\in\{11, 13, 16, 17\}$. The brute force method of enumerating all irreducible quartic plane curves over $\F_q$ is infeasible.
\end{rems}

\section{Connection with Frobenius nonclassical plane curves}\label{sect:Frob-non-classical}

In this section, we construct blocking plane curves that arise from Frobenius nonclassical curves. The concept of Frobenius nonclassical curves first naturally appeared in the work by St\"ohr and Voloch \cite{SV86} in their new proof of the Riemann hypothesis for curves over finite fields. Afterwards, Hefez and Voloch carried out a thorough investigation of these curves, and in particular, determined the precise number of $\F_q$-points on a nonsingular Frobenius nonclassical plane curve of degree $d$ \cite{HV90}*{Theorem 1}. The abundance of points on Frobenius nonclassical plane curves can be used to construct new complete arcs \cite{GPTU02}. Our approach will be similar to \cite{GPTU02}, except we are interested in using these curves to construct blocking sets instead of arcs.

While Frobenius nonclassical curves can live in a projective space of arbitrary dimension, we will primarily focus on the case of plane curves. We begin with a fundamental definition.

\begin{defn}
Suppose $C=\{F=0\}$ is a plane curve defined over $\F_q$. We say that $C$ is \emph{Frobenius nonclassical} if the polynomial $x^q F_x + y^q F_y + z^q F_z$ is divisible by $F$.
\end{defn}

Geometrically, a plane curve $C$ is Frobenius nonclassical if and only if for every non-singular point $P\in C$, the tangent line $T_P C$ contains $\Phi(P)$. Our goal is to prove Theorem~\ref{thm:Frob-nonclassical-blocking} on the existence of blocking plane curves of degree $d=o(q)$ where $q=p^n$ is a prime power with $n\geq 2$; see Remark~\ref{rem:Frob-nonclassical} for more details on the size of the degree. We begin with a general result.

\begin{prop}\label{prop:Frob-nonclassical-blocking}
Let $C$ be a smooth plane curve of degree $d$ defined over $\F_q$ where $p=\operatorname{char}(\F_q)\geq 3$. Suppose that $C$ is Frobenius nonclassical. Then $C$ is blocking.
\end{prop}

\begin{proof}
Since $C$ is a smooth Frobenius nonclassical plane curve, it follows that $C$ is non-reflexive \cite{HV90}*{Proposition 1}. Therefore, $d\equiv 1\pmod{p}$ by Pardini's theorem \cite{Par86}*{Corollary 2.2}. Applying \cite{ADL22}*{Corollary 2.3}, we obtain that $C(\F_q)$ is a blocking set. 
\end{proof}

\begin{thm}\label{thm:Frob-nonclassical-blocking}
Suppose $q=p^n$ where $n\geq 2$ and $p\geq 3$. Let $1\leq n'<n$ be a positive divisor of $n$, and set $q'=p^{n'}$. There exists a smooth blocking plane curve defined over $\F_{q}$ with degree $d=\frac{q-1}{q'-1}$.
\end{thm}

\begin{proof}
Consider the curve $C\subset \bP^2$ defined by the equation,
$$
x^{\frac{q-1}{q'-1}}+y^{\frac{q-1}{q'-1}}+z^{\frac{q-1}{q'-1}} = 0.
$$
The smoothness of the curve $C$ is clear. Moreover, $C$ is Frobenius nonclassical over $\F_q$. Indeed,
\begin{align*}
x^{q} F_x + y^{q} F_y + z^{q} F_z &= x^{\frac{q-1}{q'-1}-1+q} + y^{\frac{q-1}{q'-1}-1+q}  + z^{\frac{q-1}{q'-1}-1+q}   \\
&= x^{\frac{q'(q-1)}{q'-1}}+y^{\frac{q'(q-1)}{q'-1}}+z^{\frac{q'(q-1)}{q'-1}}
= F^{q'}.
\end{align*}
In particular, $x^{q} F_x + y^{q} F_y + z^{q} F_z$ is divisible by $F$. We conclude that $C=\{F=0\}$ is a smooth Frobenius nonclassical plane curve of degree $d=\frac{q-1}{q'-1}$. The desired result follows immediately from Proposition~\ref{prop:Frob-nonclassical-blocking}.
\end{proof}

\begin{rem}\label{rem:Frob-nonclassical} As an illustration of Theorem~\ref{thm:Frob-nonclassical-blocking}, we can let $n'=1$. We obtain a smooth blocking plane curve of degree $d=\frac{q-1}{p-1}$ over $\F_q$ for every $q=p^n$ with $n\geq 2$. For example, when $q=p^3$ with $p$ an odd prime, this yields a blocking curve of degree $d=\frac{q-1}{p-1}=p^2+p+1$. Note that $d\approx q^{2/3}$, and that the exponent $2/3$ is smaller than the exponent $3/4$ in Theorem~\ref{thm:infinitely-many-3/4-smooth}. For $n\geq 4$, we obtain a smooth blocking plane curve of degree $d\approx q^{(n-1)/n}=q^{1-1/n}$ that works for every $q$.  However, Theorem~\ref{thm:infinitely-many-3/4-smooth} has the advantage that it works in the case when $q=p$ is a prime.
\end{rem}

\begin{rem}
Let $q$ be a square. The Hermitian curve $\mathcal{H}$ given by the equation $x^{\sqrt{q}+1}+y^{\sqrt{q}+1}+z^{\sqrt{q}+1}=0$ from Example~\ref{ex:hermitian} is an example of Frobenius nonclassical plane curve over $\F_q$. In fact, $d=\sqrt{q}+1$ is the smallest possible degree of a geometrically irreducible Frobenius nonclassical plane curve over $\F_q$. Moreover, every such curve of degree $d=\sqrt{q}+1$ is $\F_q$-isomorphic to $\mathcal{H}$ \cite{BH17}*{Corollary 3.2}. 
\end{rem} 

\section{Smooth curves passing through the projective triangle}\label{sect:projective-triangle}

\begin{thm}\label{thm:projective-triangle}
Let $q\equiv 3\pmod{4}$ be a prime power with $p>3$. There exists a smooth curve $C$ of degree $d=\frac{q+3}{2}$ defined over $\F_q$ such that $C(\F_q)$ is nontrivially blocking, and $C(\F_q)$ contains the projective triangle $\Delta$.
\end{thm}

\begin{proof}
Consider the plane curve $C$ defined by the equation,
$$
xy(x^{(q-1)/2} + y^{(q-1)/2}) + yz(y^{(q-1)/2} + z^{(q-1)/2}) + zx(z^{(q-1)/2} + x^{(q-1)/2}) = 0.
$$
Note that $C$ passes through $[1:0:0]$, $[0:1:0]$, and $[0:0:1]$. Moreover, for any $s\in \F_q^{\ast}$, we have $(-s^2)^{(q-1)/2} = -1$ since $(q-1)/2$ is odd by hypothesis. As a result, the curve $C$ passes through each point of the projective triangle 
$$
\Delta = \{ [0:1:-s^2], [1:-s^2:0], [-s^2:0:1] \ | \ s\in \F_q  \}
$$
and thus $C$ is nontrivially blocking by Corollary~\ref{cor:d<q+1-nontrivial} assuming that $C$ is irreducible. It remains to show that $C$ is smooth. Note that smooth plane curves are irreducible.

The defining polynomial for $C$ can be rewritten as:
$$
F = x^{(q+1)/2}(y+z) + y^{(q+1)/2}(z+x) + z^{(q+1)/2}(x+y).
$$
Assume, to the contrary, that $C$ is singular at a point $P=[x:y:z]$. Then the three partial derivatives $F_x, F_y$ and $F_z$ must vanish at $P$. Writing down $F_x=F_y=F_z=0$ and multiplying both sides of each equation by $2$ leads to the following:
\begin{align*}
    x^{(q-1)/2}(y+z) + 2y^{(q+1)/2} + 2z^{(q+1)/2} = 0, \\
    2x^{(q+1)/2} + y^{(q-1)/2}(z+x) + 2z^{(q+1)/2} = 0, \\
    2x^{(q+1)/2} + 2y^{(q+1)/2} + z^{(q-1)/2}(x+y) = 0. 
\end{align*}
Let $m=(q-1)/2$ for simplicity. We can rewrite the above system of equations in the matrix form $M\mathbf{v}=\mathbf{0}$ where
$$
M=\begin{pmatrix} y+z & 2y & 2z  \\ 
2x & z+x & 2z \\
2x & 2y & x+y \end{pmatrix} \ \ \text{and} \ \ \mathbf{v} = 
\begin{pmatrix} 
x^{m} \\ y^{m} \\ z^{m}
\end{pmatrix}.
$$
Since $\mathbf{v}\neq\mathbf{0}$ by assumption, it follows that $\det(M)=0$. After expanding the determinant and dividing both sides by $3$ (which is permissible as $p>3$), we obtain the relation:
\begin{equation}\label{first-det-equation}
    x^2 y + x y^2 + y^2 z + y z^2 + z^2 x + z x^2 = 6xyz.
\end{equation}
The system of equations on the partial derivatives above can also be written as $N\mathbf{w}=0$ where
$$
N=\begin{pmatrix} 0 & x^m+2y^m & x^m+2z^m  \\ 
y^m+2x^m & 0 & y^m+2z^m  \\
z^m + 2x^m & z^m+2y^m & 0\end{pmatrix} \ \ \text{and} \ \ \mathbf{w} = 
\begin{pmatrix} 
x \\ y \\ z
\end{pmatrix}.
$$
Since $\mathbf{w}\neq\mathbf{0}$ by assumption, it follows that $\det(N)=0$. After expanding the determinant and factoring, we obtain the relation:
\begin{equation}\label{second-det-equation}
 6(x^m+y^m+z^m)(x^{m} y^{m} + y^{m}z^{m} + z^{m} x^{m}) = 0
\end{equation}
Since $p>3$, we can conclude that either $x^m+y^m+z^m=0$ or $x^{m} y^{m} + y^{m}z^{m} + z^{m} x^{m}=0$.

\textbf{Case 1.} $x^m+y^m+z^m=0$.

In this case, we can substitute $x^m=-y^m-z^m$ into the equation  $x^m(y+z)+2y^{m+1}+2z^{m+1}=0$ given by the vanishing of the partial derivative $F_x$. We obtain,
\begin{equation}
\label{eq:1009}
-(y^m+z^m)(y+z)+2y^{m+1}+2z^{m+1} = (y-z)(y^m-z^m) = 0.
\end{equation}
Equation~\eqref{eq:1009} yields that $y^m=z^m$. By symmetry, we can apply the same argument to the equations given by $F_y=0$ and $F_z=0$ to get $z^m=x^m$, and also, $x^m=y^m$. We can then conclude that $x^m=y^m=z^m$, and thus $x^m+y^m+z^m=0$ implies that $3x^m=0$. Since $p>3$, we must have  $x=y=z=0$, which is a contradiction.

\textbf{Case 2.} $x^m y^m +y^m z^m +z^m x^m=0$.

We split the analysis into two sub-cases.

\textbf{Case 2.1.} $xyz=0$.

Without loss of generality, suppose $z=0$. Then the equation $F_z=0$ becomes:
\begin{equation}\label{eq:partial-derivative-z}
    2x^{m+1}  + 2y^{m+1} = 0.
\end{equation}
Now, combining $z=0$ and $x^m y^m +y^m z^m +z^m x^m=0$, we must have $x^m y^m=0$, that is, either $x=0$ or $y=0$. However, both possibilities imply $x=y=0$ using \eqref{eq:partial-derivative-z}, which is a contradiction, since at least one of $x, y, z$ must be non-zero.

\textbf{Case 2.2.} $xyz\neq 0$.

We introduce new variables $A = 1/x, B = 1/y$ and $C = 1/z$. The relation $x^m y^m +y^m z^m +z^m x^m=0$ implies 
\begin{equation}\label{ABC-relation-1}
A^m + B^m + C^m = 0.
\end{equation}
The defining equation for the curve can be expressed as,
$$
\frac{B+C}{A^{m+1}BC} + \frac{C+A}{B^{m+1}CA} + \frac{A+B}{C^{m+1} AB} = 0.
$$
After multiplying both sides by $ABC$, we get
$$
\frac{B+C}{A^{m}} + \frac{C+A}{B^{m}} + \frac{A+B}{C^{m}} = 0.
$$
The last equation can be rewritten as,
$$
C\cdot \left( \frac{1}{A^m} + \frac{1}{B^m} \right) + B\cdot \left(\frac{1}{C^m} + \frac{1}{A^m}\right) + A\cdot \left(\frac{1}{B^m} + \frac{1}{C^m}\right) = 0.
$$
Combining the last relation with \eqref{ABC-relation-1}, 
\begin{align*}
0 &= C\cdot \left( \frac{A^m + B^m}{(AB)^m} \right) + B\cdot \left(\frac{C^m +A^m}{(CA)^m} \right) + A\cdot
\left(\frac{B^m+C^m}{(BC)^m}\right)  \\
&= C\cdot \frac{(-C^m)}{(AB)^m} + B\cdot \frac{(-B^m)}{(CA)^m} + A\cdot
\frac{(-A^m)}{(BC)^m}.
\end{align*}
Next, we multiply the last equation by $(ABC)^m$ to arrive to,
$$
0 = C^{2m+1} + B^{2m+1} + A^{2m+1}.
$$
Recalling that $m=\frac{q-1}{2}$, we have $2m+1=q$. Thus, 
$$
0 = C^q + B^q + A^q = (A+B+C)^q
$$
as we are working in characteristic $p$. We conclude that $A+B+C=0$, or equivalently, $xy+yz+zx=0$. In particular, we must have $(x+y+z)(xy+yz+zx)=0$, that is,
\begin{align}\label{new-cubic-relation}
x^2 y + x y^2 + y^2 z + y z^2 + z^2 x + z x^2 + 3xyz = 0.
\end{align}
Finally, combining \eqref{first-det-equation} and \eqref{new-cubic-relation}, we conclude that
$$
6xyz+3xyz = 0 \ \Rightarrow \ 9xyz = 0 \ \Rightarrow \ xyz=0
$$
as we are assuming $p>3$. This contradicts the assumption that $xyz\neq 0$.

We deduce that the curve $C$ is smooth, and the proof is complete.
\end{proof}

\begin{rem}\label{rem:bezout} 
Note that $(q+3)/2$ is a lower bound on the degree of an irreducible curve that passes through all the points of the projective triangle. Indeed, the intersection between such a curve and the line $x=0$ includes the points $\{[0:1:-s^2] \ | \ s \in \F_q^{\ast}\} \cup \{[0:1:0],[0:0:1]\}$. Since there are at least $(q-1)/2+2=(q+3)/2$ intersection points, the degree must be at least $(q+3)/2$ by B\'ezout's theorem. We have shown that this is tight when $q \equiv 3 \pmod 4$.  

Note that $(p+3)/2$ is also a lower bound on the degree of an irreducible curve that passes through a blocking set of R\'edei type over $\F_p$. A blocking set $B$ is of \emph{R\'edei type} if there is a line $L$ such that $|B \cap L|=|B|-q$. For example, the projective triangle is of R\'edei type. Blokhuis \cite{B94} showed each nontrivial blocking set in $\bP^2(\F_p)$ has size at least $3(p+1)/2$, and so a nontrivial blocking set of R\'edei type in $\bP^2(\F_p)$ contains a line $L$ with at least $3(p+1)/2-p=(p+3)/2$ points.  In fact, G\'acs \cite{G03}  showed that if a blocking set of R\'edei type in $\bP^2(\F_p)$ is not projectively equivalent to the projective triangle, then it has size at least $p+2(p-1)/3+1$. It follows that any irreducible plane curve passing through a R\'edei type blocking set other than the projective triangle (and its image under a projective transformation) must have degree at least $2p/3$.
\end{rem}

\begin{rem}
Note that our proof relies on the fact that $-1$ is a non-square in $\F_q$ provided that $q \equiv 3 \pmod 4$. When $q \equiv 1 \pmod 4$, $-1$ is a square in $\F_q$, and one can instead consider the curve given by,
$$xy(x^{(q-1)/2} - y^{(q-1)/2}) + yz(y^{(q-1)/2} - z^{(q-1)/2}) + zx(z^{(q-1)/2} - x^{(q-1)/2})=0$$
which does pass through the points of the projective triangle, and hence is a blocking curve. However, the curve above is reducible (in fact, contains the lines $x=y$, $y=z$ and $z=x$).  Nonetheless, we believe that there is a smooth degree $\frac{q+3}{2}$ curve that passes through the points of the projective triangle when $q\equiv 1\pmod{4}$. In fact, for every homogeneous polynomial $g(x,y,z)$ of degree $\frac{q-3}{2}$, any plane curve defined by
$$
xy(x^{(q-1)/2} - y^{(q-1)/2}) + yz(y^{(q-1)/2} - z^{(q-1)/2}) + zx(z^{(q-1)/2} - x^{(q-1)/2}) + xyz \cdot g(x,y,z)=0
$$
passes through the projective triangle, and thus, is a blocking curve. The subtle point is to find a suitable $g$ to ensure that the curve is smooth. We believe that already taking $g(x,y,z)=x^{(q-3)/2}$ would produce a smooth curve, but we were unable to prove this for all $q\equiv 1\pmod{4}$.
\end{rem}

\section{Constructions of irreducible and smooth blocking curves}\label{sect:constructions-I}

The fact that the projective triangle is a blocking set relies on nice properties of squares and non-squares. Inspired by this observation, we provide a systematic approach to constructing blocking curves using power residues in the following proposition.  

\begin{prop} \label{prop: general-construction}
Let $r\geq 2$ be a positive integer. Let $q\equiv 1\pmod{r}$ be a prime power such that $q>r^4$. Let $f,g,h \in \F_q[x,y,z]$ be homogeneous polynomials with the same degree, such that $f=-g-h$, $f(1,0,0)=0, g(0,1,0)=0$, and $h(0,0,1)=0$.  Consider the curve $C$ defined by
$$
f(x,y,z)x^m+g(x,y,z)y^m+h(x,y,z)z^m=0,
$$ 
where $m=(q-1)/r$. Then $C$ is blocking.  
\end{prop}
\begin{proof}
By hypothesis, it is clear that $C$ passes through the points $[0:0:1], [0:1:0], [1:0:0]$. Moreover, since $f=-g-h$, the curve $C$ passes through $[1:y:z]$ whenever $y$ and $z$ are $r$-th powers in $\F_q^*$. Thus, it suffices to show that
$$
B=\{[0:0:1], [0:1:0], [1:0:0]\} \cup \{[1:y:z] \ | \text{ $y$ and $z$ are $r$-th powers in $\F_q^*$} \}
$$
forms a blocking set.

Consider an $\F_q$-line $L\colon ax+by+cz=0$. If $abc=0$, then the line $L$ contains a point in $\{[0:0:1], [0:1:0], [1:0:0]\}$. Next assume that $abc \neq 0$. We may further assume that $a=1$. To show $B \cap L\neq \emptyset$, it suffices to show that $by+cz=-1$ holds for some $y$ and $z$ that are $r$-th powers in $\F_q^*$. This is guaranteed by Corollary~\ref{cor:by+cz}.
\end{proof}

Nonetheless, it is not clear which curves in the family given by the previous proposition are smooth or irreducible. By specializing the choice of $f, g, h$, we construct a geometrically irreducible blocking curve in every degree starting from $d=\frac{q-1}{r}+1$.

\begin{thm}\label{thm:irreducible-curves}
Let $r\geq 2$ be a positive integer. Let $q\equiv 1\pmod{r}$ be a prime power such that $q>r^4$. Let $k\in\F_q^*$ such that $-k$ is not an $r$-th power in $\F_q^*$. Let $m=\frac{q-1}{r}$. Then the curve $C$ defined by
$$
-(ky^{\ell}+z^{\ell})x^m + z^{\ell} y^m + k y^{\ell} z^m =0
$$
is geometrically irreducible and nontrivially blocking for each $\ell\geq 1$.
\end{thm}

\begin{proof}
We first show that $\gcd(ky^{\ell}+z^{\ell}, z^{\ell} y^m + k y^{\ell} z^m)=1$. It suffices to show that there are no $y,z \neq 0$ in $\overline{\F_q}$ such that $ky^{\ell}+z^{\ell}=0$ and $y^{m-\ell}+kz^{m-\ell}=0$ hold at the same time. Otherwise, we would have,
$$
z^{\ell}=-ky^{\ell}, z^{m-\ell}=(-k)^{-1}y^{m-\ell}.
$$
This implies that 
$$
(-k)^{m-\ell} y^{\ell(m-\ell)}=z^{\ell(m-\ell)}=(-k)^{-\ell} y^{\ell(m-\ell)},
$$
thus $(-k)^m=1$, which implies that $-k$ is an $r$-th power in $\F_q^*$, violating the assumption.

Next, we show that $C$ is geometrically irreducible. Since $p\nmid m=(q-1)/r$, we know that either $p\nmid \ell$ or $p\nmid (m-\ell)$. In either case, we can apply Lemma~\ref{lem:eisenstein-criterion} to the polynomial 
$$
-(ky^\ell+z^\ell)x^m + (z^\ell y^m + ky^\ell z^m)
$$
seen over the field $\overline{\mathbb{F}_q}$. Indeed, $z^\ell y^m + ky^\ell z^m$ has a non-repeated factor if $p\nmid (m-\ell)$ and $ky^\ell+z^\ell$ has a non-repeated factor if $p\nmid \ell$. 

By Proposition~\ref{prop: general-construction}, $C$ is  blocking. Thus, it remains to show that $C$ is nontrivially blocking; note that since $C$ has potentially large degree, Corollary~\ref{cor:d<q+1-nontrivial} cannot be applied and it may still contain all the $\F_q$-points of some $\F_q$-line despite being geometrically irreducible. However, observe that the curve $C$ does not pass through $[1:y:z]$ where $y$ is an $r$-th power and $z$ is not an $r$-th power, for otherwise $-(ky^\ell+z^\ell)+z^\ell+ky^\ell z^m=0$, which implies that $y=0$ since $z^m \neq 1$, a contradiction. 

Now, suppose that the equation of a line $L$ is given by $ax+by+cz=0$ where $a, b, c\in\F_q$; we need to show there is an $\F_q$-point on $L$ which is not on $C$.

If $abc\neq 0$, we can assume $a=1$. Then any point $[1:y:z]$ which lies on $L$ satisfies $by+cz=-1$. We can apply Corollary~\ref{cor:by+cz} to find some $r$-th power $y$ and some non-$r$-th power $z$ which satisfies this relation; but such a point does not lie on the curve $C$ by the above discussion. It remains to analyze the case $abc=0$. 

When $a=0$, $b=0$ and $c\neq 0$, then the line $L$ is given by $z=0$. In this case, $L\cap C$ only consists of two points $[1:0:0]$ and $[0:1:0]$. Similarly, when $a=0$, $c=0$ and $b\neq 0$, then $L=\{y=0\}$ and $L\cap C$ only consists of $[1:0:0]$ and $[0:0:1]$.

When $a=0$ and $bc\neq 0$, the equation of $L$ is given by $by+cz=0$, we get $y=(-c/b)z$. There are two cases to consider. 

\textbf{Case 1.} $-c/b$ is not a $r$-th power in $\F_q$. \\
Then $L$ contains the point $[1:1:-c/b]$. Since $1$ is an $r$-th power and $-c/b$ is a non-$r$-th power in $\F_q$, such a point is not contained in $C$. 

\textbf{Case 2.}  $-c/b$ is an $r$-th power in $\F_q$. \\
Then $L$ contains a point $[1:y_0:z_0]$ with $y_0^m=z_0^m=\omega$ where $\omega\notin\{0,1\}$. We claim that this $\F_q$-point is not on $C$. Otherwise, 
$$
-(ky_0^\ell + z_0^\ell)+z_0^\ell y_0^m+ky_0^\ell z_0^m =(\omega-1)(ky_0^\ell + z_0^\ell)= 0.
$$
This forces $k y_0^\ell+z_0^\ell=0$, and also $z_0^\ell y_0^m+ky_0^\ell z_0^m=0$. This contradicts the earlier statement that $\gcd(z^\ell + ky^\ell, z^\ell y^m+ky^\ell z^m)=1$.

When $a \neq 0$ and $b=c=0$, then line $L$ is given by $x=0$. Note that the point $[0:1:1]$ is on the line $x=0$ but not on the curve $C$, since $-k$ is a non-$r$-th power implies that $k \neq -1$.
 
Finally we consider the case  $a\neq 0$, and exactly one of $b=0$ or $c=0$ holds. Observe that the curve does not pass through $[1:y:0]$ for $y \neq 0$ and $[1:0:z]$ for $z \neq 0$. If $b=0$, then $L$ contains $[1:0:z_0]$ for some nonzero $z_0$ and if $c=0$ then $L$ contains $[1:y_0:0]$ for some nonzero $y_0$. 
\end{proof}

One can check that when $\ell \geq 2$, the curve in Theorem~\ref{thm:irreducible-curves} is singular at the points $[1:0:0],[0:1:0],[0:0:1]$. Next, we will show that when $\ell=1$, the above construction in fact gives a smooth blocking curve.

\begin{thm}\label{thm:smoothcurve}
Let $r\geq 2$ be a positive integer. Let $q\equiv 1\pmod{r}$ be a prime power such that $q>r^4$ and $p \nmid (r^2-1)$. Let $k\in\F_q^*$ such that $-k$ is not an $r$-th power in $\F_q^*$, and set $m=\frac{q-1}{r}$. Then the curve $C$ defined by
$$
F(x,y,z)=-(ky+z)x^{m}+ z y^m + ky z^m = 0
$$
is smooth and nontrivially blocking.
\end{thm}
\begin{proof}
We have shown in Theorem~\ref{thm:irreducible-curves} that $C$ is nontrivially blocking. It suffices to show that $C$ is smooth.

Suppose $P=[x:y:z]$ is a singular point of $C$. The conditions $F_x(P)=0$, $F_y(P)=0$ and $F_z(P)=0$ become, after multiplying both sides by $r$,
\begin{align*}
    & (ky+z)x^{m-1}  = 0, \\
    & -rkx^m -z y^{m-1} + rkz^{m} = 0, \\
    & -rx^{m} + ry^{m}  -kyz^{m-1} = 0. 
\end{align*}
We can rewrite the system of equations in the matrix form $M\mathbf{v}=\mathbf{0}$ where
$$
M=\begin{pmatrix} 
ky+z & 0 & 0  \\ 
-rkx & -z & rkz \\
-rx & ry & -ky \end{pmatrix} \ \ \text{and} \ \ \mathbf{v} = 
\begin{pmatrix} 
x^{m-1} \\ y^{m-1} \\ z^{m-1}
\end{pmatrix}.
$$
Since $\mathbf{v}\neq\mathbf{0}$ by assumption, it follows that $\det(M)=0$:
$$
(r^2-1)k(ky+z)yz=0.
$$
Also note that $F_x=0$ implies that $(ky+z)x^{m-1}=0$.

\textbf{Case 1:} $x \neq 0$. Then we must have $ky+z=0$. Since $F=0$, we have $$0=zy^m+kyz^m=z(y^m-z^m).$$
If $z=0$, then $y=0$ and thus $F_y=-rkx^{m} \neq 0$ since $p \nmid r$, a contradiction. Thus, $z \neq 0$ and $y \neq 0$, and we must have 
$$
y^m=z^m=(-ky)^m=(-k)^m y^m,
$$
which implies that $-k$ is an $r$-th power, violating our assumption.

\textbf{Case 2:} $x=0$. Since $F_y=F_z=0$, we have
$$
rkz^m=zy^{m-1}, ry^m=kyz^{m-1}.
$$
Thus, $y \neq 0, z \neq 0$. It follows that $ky+z=0$ since $\det(M)=0$. We can now argue in the same way as we did in \textbf{Case 1} to deduce that $-k$ is an $r$-th power, which is a contradiction. 

We conclude that $C$ is a smooth curve. 
\end{proof}

We discuss the sharpness of the assumption $q>r^4$ in the statement of Theorem~\ref{thm:smoothcurve} for small values of $r$ below.

\begin{rem}\label{rmk:bound-q>r^4}
When $r=2$, the hypothesis requires $q > 16$. However, one can show that $q\geq 7$ is already sufficient by analyzing the cases $q\in\{7, 11, 13\}$ in a computer. On other hand, the conclusion fails when $q=5$, because the degree is $d=\frac{q-1}{r}+1=3$, and an irreducible cubic curve over $\F_5$ cannot be blocking by Theorem~\ref{thm:cubic}.

When $r=3$, the hypothesis requires $q>3^4=81$. It turns out that the conclusion fails when $q=13$. However, we checked with a computer that the conclusion of the theorem holds for all $19\leq q\leq 81$. 

When $r=4$, the hypothesis requires $q>4^4=256$. One can check that the conclusion of Theorem~\ref{thm:smoothcurve} does not hold when $q=29$. On the other hand, the conclusion holds for $q=37$. We believe that the conclusion holds for all $q\geq 37$.

When $r=5$, the hypothesis requires $q>5^4=625$. One can check using a computer that the conclusion of Theorem~\ref{thm:smoothcurve} does not hold when $q=101$. On the other hand, the conclusion of Theorem~\ref{thm:smoothcurve} holds for $q=131$. We believe that the conclusion holds for all $q\geq 131$. 

In general, we believe the bound $q>r^4$ is not optimal. However, to the best knowledge of the authors, relaxing the inequality $q>r^4$ in Corollary~\ref{cor:by+cz} is out of reach.
\end{rem}

We end the section by presenting the proof of Theorem~\ref{thm:infinitely-many-3/4} and Theorem~\ref{thm:infinitely-many-3/4-smooth}. 

\begin{proof}[Proof of Theorem~\ref{thm:infinitely-many-3/4}]
Applying Corollary~\ref{cor:divisor} with $\theta=1/4$ and $A=1/2$, we can find infinitely many primes $p$ such that $p-1$ has a divisor $r$ such that $\frac{1}{4} p^{1/4} \leq r <p^{1/4}$. Note that for each such a pair $(p,r)$, we have $p>r^4$, $p \equiv 1 \pmod r$, and $(p-1)/r \leq 4p^{3/4}$. For each such a pair $(p,r)$,  Theorem~\ref{thm:irreducible-curves} implies that there is a geometrically irreducible nontrivially blocking curve over $\F_p$ with degree $d$ for each choice of $d\geq \frac{p-1}{r}+1$. Since $(p-1)/r \leq 4p^{3/4}$, we obtain desired curves in every degree starting with $4p^{3/4}+1$.
\end{proof}

\begin{proof}[Proof of Theorem~\ref{thm:infinitely-many-3/4-smooth}]
The proof is similar to the proof of Theorem~\ref{thm:infinitely-many-3/4}. It follows from Corollary~\ref{cor:divisor} and Theorem~\ref{thm:smoothcurve}. When $\theta=1/4$, the requirement $A\leq 1/2$ is imposed by the bound $q>r^4$ which appears as a hypothesis in Theorem~\ref{thm:smoothcurve}.
\end{proof}

\section{Smooth blocking curves with degree \texorpdfstring{$(q+3)/2$}{(q+3)/2} and \texorpdfstring{$(q-1)/2$}{(q-1)/2}}\label{sect:constructions-II}

\subsection{Proof of Theorem~\ref{thm:construction-smooth-(q+3)/2}}

When $q \equiv 3 \pmod 4$, we have already proved the existence of such a curve in Theorem~\ref{thm:projective-triangle}. So, we can assume $q \equiv 1 \pmod 4$ and $p>3$ for the rest of the proof.

Let $m=(q-1)/2$. We consider the plane curve $C$ defined by
$$
F(x,y,z)= xy (x^m + y^m) + (xz+yz)(x^m+z^m)=0.
$$
We claim that $C$ is smooth and nontrivially blocking.

When $q=5,13$, one can check $C$ is smooth and nontrivially blocking using a computer. Next, we assume $q \geq 17$ so that Corollary~\ref{cor:by+cz} can be applied with $r=2$. Note that we can write $$F(x,y,z)=(xy+xz+yz)x^m+xy \cdot y^m+(xz+yz)z^m,$$
so $C$ contains the points in 
$$
\{[0:0:1], [0:1:0], [1:0:0]\} \cup \{[1:y:z] | \text{ $y$ and $z$ are non-squares in $\F_q^*$} \}.
$$
We can argue as in Proposition~\ref{prop: general-construction} and Corollary~\ref{cor:d<q+1-nontrivial} that $C$ is nontrivially blocking. Thus, it suffices to show that $C$ is smooth.

Suppose $P=[x:y:z]$ is a singular point of $C$. The conditions $F_x(P)=0$, $F_y(P)=0$ and $F_z(P)=0$ become, after multiplying both sides by $2$,
\begin{align*}
    & x^{m} y + 2y^{m+1} + x^m z + 2 z^{m+1} - x^{m-1} yz = 0, \\
    & 2x^{m+1} + x y^{m} + 2 x^m z + 2z^{m+1} = 0, \\
    & 2x^{m+1} + x z^{m} + 2 x^m y + y z^m = 0. 
\end{align*}
We can rewrite the system of equations in the matrix form $M\mathbf{v}=\mathbf{0}$ where
$$
M=\begin{pmatrix} xy+xz-yz & 2y^2 & 2z^2  \\ 
2x^2+2xz & xy & 2z^2 \\
2x^2+2xy & 0 & xz+yz \end{pmatrix} \ \ \text{and} \ \ \mathbf{v} = 
\begin{pmatrix} 
x^{m-1} \\ y^{m-1} \\ z^{m-1}
\end{pmatrix}.
$$
Since $\mathbf{v}\neq\mathbf{0}$ by assumption, it follows that $\det(M)=0$. After expanding the determinant and factoring, we obtain:
$$
\det(M) = (-3)\cdot xyz \cdot (x+y)\cdot (xy+xz-yz)=0.
$$
Since $\operatorname{char}(\F_q)>3$ by hypothesis, we may ignore the $(-3)$ factor. We will proceed according to which factor above vanishes. 

\textbf{Case 1.} $xyz=0$.

We have three subcases to consider. 

\textbf{Case 1.1.} $z=0$. \\
In this case, $xy(x^m+y^m)=0$ using the defining equation of the curve. If $x=0$, then $y\neq 0$ but then $F_x \neq 0$, a contradiction. If $y=0$, then $x\neq 0$ but then $F_{z}\neq 0$, a contradiction. Thus, $xy\neq 0$ which means $x^m+y^m=0$. Now, $F_x=0$ gives $x^m y + 2y^{m+1}=0$. We obtain $0=y(-y^{m})+2y^{m+1}=y^{m+1}$, which forces $y=0$, a contradiction.

\textbf{Case 1.2.} $y=0$. \\
In this case, $xz(x^m+z^m)=0$ using the defining equation of the curve. If $z=0$, then we have handled this in Case 1.1. So, we may assume that $z\neq 0$. If $x=0$, then $F_y\neq 0$, a contradiction. We must have $xz\neq 0$, which means $x^m+z^m=0$. From $F_x=0$, we get $x^m z + 2 z^{m+1}=0$. Thus, $0=x^m z + 2z^{m+1} = z(-z^{m})+2z^{m+1}=z^{m+1}$. But then $z=0$, which has already been handled in Case 1.1.

\textbf{Case 1.3.} $x=0$. 

If $x=0$, then $F_z=0$ implies that $yz^m=0$, and so $z=0$ or $y=0$. We have handled these in Case 1.1 and Case 1.2, respectively.

\textbf{Case 2.} $xyz\neq 0$ and $x+y=0$.

In this case, $y=-x$ and the defining equation for the curve becomes $-x^2(x^{m}+(-x)^{m})=0$. Since $m=(q-1)/2$ is even by hypothesis, we get $-x^2\cdot (2x^{m})=0$, and so $x=0$, a contradiction. 

\textbf{Case 3.} $xyz\neq 0$, $x+y\neq 0$, and $xy+xz-yz=0$.

Looking at the first row of the matrix equation $M\mathbf{v}=0$, we obtain $2y^{m+1}+2z^{m+1}=0$, that is, $y^{m+1}+z^{m+1}=0$. Let us write $F=0$ more explicitly:
$$
x^{m+1} y + x y^{m+1} + x^{m+1} z + x z^{m+1} + y z x^{m} + y z^{m+1} = 0 
$$
and rearrange to get:
$$
\underbrace{x^{m+1}(y+z) + yz x^{m}}_{=x^{m}(xy+xz+yz)=x^{m}(2yz)} + x \underbrace{(y^{m+1} + z^{m+1})}_{=0} + y z^{m+1} = 0.  
$$
We deduce that:
$$
2x^m yz + y z^{m+1}=0 \ \ \Rightarrow \ \ 2x^m + z^m = 0. 
$$
We will now combine the three relations $xy+xz=yz$,  $y^{m+1}+z^{m+1}=0$, and $2x^m+z^m=0$. Since we are working in projective coordinates and $xyz\neq 0$, without loss of generality, we can set $z=1$. These three relations become:
\begin{align*}
    xy &= y-x, \\
    y^{m+1} &= -1, \\
    2x^m &= -1.
\end{align*}

Since $y^{m+1} = -1$, we square both sides to get $y^{2m+2} = 1$. But $2m+1 = q$ by hypothesis, and so $y^{q+1} = 1$. In particular, $y^{q^2-1} = 1$ which means $y$ belongs to $\F_{q^2}$. Since $x$ and $y$ are related by $xy = y-x$, this means $x$ is in $\F_{q^2}$ as well, and so $x^{q^2-1}=1$ too.

On the other hand, since $2x^{m}=-1$, that is, $2 x^{(q-1)/2} = -1$, this would imply after squaring both sides: $4 x^{q-1} = 1$, and so $4^{q+1} x^{q^2-1} = 1$. Since $x^{q^2-1}=1$, this last equation means $4^{q+1} = 1$. From here, we can use $4^q = 4$ to get $4^2 = 1$ in $\F_q$, that is, $15 = 0$ holds in $\F_q$. Since the characteristic $p$ is bigger than $3$, we conclude $p=5$. 

In this case, $4x^{q-1}=1$ implies $x^{q-1}=-1$. Now, $xy = y-x$ implies $x^q y^q = y^q - x^q$. Using $x^{q-1}=-1$ and $y^{q+1}=1$, this last equation can be written as $- x/y = (1/y) + x$ or equivalently, $-x = 1 + xy$. But we know $xy=y-x$, and so $-x = 1 + y - x$ which forces $y=-1$. Substituting this back again into $xy=y-x$, we get $-x = -1 - x$, which is a contradiction. We deduce that $C$ is smooth.

\begin{rem}
When $q \equiv 3 \pmod 4$, and $p>3$, one can use a similar argument to show that the curve $C$ defined by
$$
F(x,y,z)=-(xy+xz+yz)x^m+xy \cdot y^m+(xz+yz)z^m=0
$$
is smooth and nontrivially blocking, where $m=(q-1)/2$.
\end{rem}

\subsection{Proof of Theorem~\ref{thm:construction-smooth-(q-1)/2}}

Let $q\equiv 3 \pmod{4}$ such that $p>3$ and $q\equiv 3 \pmod{4}$. Let $m=(q-1)/2$. Consider the plane curve $C$ defined by
$$
(x+y+z)^m + x^m + y^m + z^m = 0.
$$
We show that $C$ is smooth and nontrivially blocking in a series of two claims. Note that the proof of Claim~\ref{claim:smoothness} works for all prime powers $q=p^r$ with $p>3$. 

\begin{claim}\label{claim:smoothness}
$C$ is smooth.
\end{claim}

\begin{proof} Assume, to the contrary, that $P=[x:y:z]$ is a singular point of $C$. Let $F=(x+y+z)^{m}+x^m+y^m+z^m$ be the defining polynomial of $C$. The conditions $F_x(P)=F_y(P)=F_z(P)=0$ translate to:
\begin{align*}
m(x+y+z)^{m-1}+m x^{m-1} = 0, \\
m(x+y+z)^{m-1}+m y^{m-1} = 0, \\
m(x+y+z)^{m-1}+m z^{m-1} = 0. 
\end{align*}
Thus, $x^{m-1}=y^{m-1}=z^{m-1}=-(x+y+z)^{m-1}$. If $xyz=0$, then it is clear that $x=y=z=0$ which would be a contradiction. So, we may assume that $xyz\neq 0$. The relations tell us that there are $\alpha, \beta \in \overline{\F_q}$ such that $x=\alpha z$ and $y=\beta z$, where $\alpha^{m-1}=\beta^{m-1}=1$. Since $m-1=\frac{q-3}{2}$, we get $\alpha^{q-3}=1$. And so $\alpha^q = \alpha^3$. Similarly, $\beta^q=\beta^3$. Using $z^{m-1}=-(x+y+z)^{m-1}$, we obtain 
$$
(1+\alpha+\beta)^{m-1} = -1 \ \ \Rightarrow \ \ (1+\alpha+\beta)^{q-3} = 1. 
$$
Therefore,
$$
1+\alpha^q + \beta^q = (1 + \alpha + \beta)^q = (1+\alpha+\beta)^{3}.
$$
Consequently, 
\begin{equation}
\label{eq:1010}
1 +\alpha^q + \beta^q = 1 + \alpha^3 + \beta^3 + 3(\alpha\beta + 1)(\alpha +
\beta) + 3(\alpha + \beta)^2.
\end{equation}
Using $\alpha^q=\alpha^3$ and $\beta^q=\beta^3$, the equation~\eqref{eq:1010} simplifies to:
$$
3(\alpha+\beta)(\alpha+1)(\beta+1)=0.
$$
If $\alpha=-1$, then $x=-z$, but then $y^{m-1}=-(x+y+z)^{m-1}$ implies that $y^{m-1}=-y^{m-1}$, contradicting $xyz\neq 0$. If $\beta=-1$, then a similar calculation reaches $x^{m-1}=-x^{m-1}$, a contradiction. If $\alpha+\beta=0$, then $\alpha=-\beta$, which means $y=-x$, and the same reasoning as above shows $z^{m-1}=-z^{m-1}$, which again gives a contradiction. Therefore, $C$ is a smooth curve.  
\end{proof}

\begin{claim}\label{claim:non-trivially-blocking}
$C$ is nontrivially blocking.
\end{claim}
 
\begin{proof}
By Corollary~\ref{cor:d<q+1-nontrivial}, it suffices to show that $C$ is blocking. When $q=7$, the degree is $d=\frac{q-1}{2}=3$, and the conclusion fails because an irreducible cubic curve over $\F_7$ cannot be blocking by Theorem~\ref{thm:cubic}. When $11\leq q<47$, it is easy to verify the statement using a computer. Next, we assume that $q \geq 47$ so that we can apply Corollary~\ref{cor:47}.

Let $L$ be an $\F_q$-line $ax+by+cz=0$ in $\mathbb{P}^2$. We consider several cases depending on the values of $a, b, c$. The convention for numbering different cases follows a tree structure.

\textbf{Case 0.} $abc=0$. Since the equation of the curve is symmetric in $x,y,z$, we may assume $c=0$. We have two subcases to consider.

\textbf{Case 00.} $b=0$. In this case, $L$ is given by $\{x=0\}$, and $C\cap L$ contains the point $[0:1:-1]$. We have used $(-1)^{m}=-1$ which holds due to $q\equiv 3\pmod{4}$.

\textbf{Case 01.} $a=0$. In this case, $L$ is given by $\{y=0\}$, and $C\cap L$ contains the point $[1:0:-1]$.

\textbf{Case 02.} $ab\neq 0$. We can assume $b=1$ for the rest of this case. 

\textbf{Case 020.} $a=1$. In this case, $L$ is given by $x+y=0$, and $L\cap C$ contains $[1:-1:0]$.

\textbf{Case 021.} $a\neq 1$. The equation for $L$ is $y=-ax$. Substituting this into the equation of $C$, we get:
$$
x^m - a^m x^m + z^m + ((1-a)x + z)^m = 0.
$$
We will look for a solution with $z=1$, in which case the equation above becomes:
\begin{align}\label{eq:Case021}
x^m - a^m x^m + 1 + ((1-a)x + 1)^m = 0.
\end{align}

\textbf{Case 0210.} $a^m = 1$. 
In this case, \eqref{eq:Case021} becomes $((1-a)x + 1)^m = -1$. We can always find  $x_0\in\F_q^*$ such that $(1-a)x_0+1$ is a non-square since $a \neq 1$. Such a point satisfies $((1-a)x_0 + 1)^m = -1$, and therefore $[x_0:-ax_0:1]\in L\cap C$.

\textbf{Case 0211.} $a^m=-1$. In this case, \eqref{eq:Case021} becomes $2x^m + 1 + ((1-a)x + 1)^m = 0$. By Corollary~\ref{cor:47}, we can find a non-square $x_0 \in \F_q$ such that $(1-a)x_0+1$ is a nonzero square. Such a point satisfies $x_0^m=-1$ and $((1-a)x_0 + 1)^m = 1$. Therefore, $[x_0:-ax_0:1]\in L\cap C$.

This concludes the analysis of the case $abc=0$. We move on to the next main case.

\textbf{Case 1.} $abc\neq 0$. 

We work under the assumption now that $c=1$. Again, there are several cases to consider.

\textbf{Case 10.} $a=b=1$. In this case, $C\cap L$ contains the point $[1:-1:0]$.

Again, using the symmetry of the curve, we will proceed according to the following case.

\textbf{Case 11}. $b\neq 1$. There are two further sub-cases to consider:

\textbf{Case 110.} $a=1$. The equation of $L$ is $z=-(x+by)$. The intersection between $L$ and $C$ is thus governed by the following relation:
\begin{equation}\label{eq:Case110}
    x^m + y^m - (x + by)^m + ((1-b)y)^m = 0.
\end{equation}

\textbf{Case 1100.} $1-b$ is a (nonzero) square. In this case, \eqref{eq:Case110} leads to:
$$
x^m + 2y^m - (x+by)^m = 0.
$$
By setting $y=1$, it suffices to find $x_0\in \F_q$ such that $x_0$ is a non-square and $x_0+b$ is a nonzero square.  Such $x_0$ exists by Corollary~\ref{cor:47}.

\textbf{Case 1101.} $1-b$ is a non-square. In this case, \eqref{eq:Case110} leads to:
$$
x^m = (x+by)^m.
$$
By setting $y=1$, it suffices to find an $x_0\in\F_q$ such that $x_0$ and $x_0+b$ are both nonzero squares. Such $x_0$ exists by Corollary~\ref{cor:47}.

\textbf{Case 111.} $a\neq 1$ (note that we already know that $b\neq 1$). The equation of the line $L$ is $z=-ax-by$. Substituting this into the equation of $C$, the points in $L\cap C$ are determined by:
$$
x^m + y^m - (ax + by)^m + ((1-a)x + (1-b)y)^m = 0. 
$$
We look for solutions to the above equation when $x, y\in\F_q$. Setting $y=1$, it suffices to analyze:
\begin{align}\label{eq:Case111}
x^m - (ax + b)^m + ((1-a)x + 1-b)^m = -1.
\end{align}
We proceed according to two subcases, depending on whether $a$ is equal to $b$.

\textbf{Case 1110.} $a=b$. In this case, the equation \eqref{eq:Case111} is satisfied when $x=-1$ because $q\equiv 3\pmod{4}$. Thus, $L\cap C$ certainly contains $\F_q$-points in this case.

\textbf{Case 1111.} $a\neq b$. In this case, in order to satisfy \eqref{eq:Case111}, it suffices to find $x_0\in \F_q$ such that $x_0$ is a non-square, but $ax_0+b$ and $(1-a)x_0 + 1-b$ are both nonzero squares. After dividing both sides by $a$ and $1-a$, we need to ensure that $x_0$ is a non-square, $x_0+\frac{b}{a}$ and $x_0 + \frac{1-b}{1-a}$ have prescribed form (squares or non-squares depending on $a$ and $1-a$). Such an $x_0$ exists by Corollary~\ref{cor:47} since $a \neq b$ implies that $\frac{b}{a} \neq \frac{1-b}{1-a}$.

This concludes the proof that $C(\F_q)$ is a blocking set.
\end{proof}

Note that the condition $q\equiv 3\pmod{4}$ was used in the proof of the previous theorem. Indeed, the given curve will not be blocking when $q\equiv 1\pmod{4}$ and $p>3$, because the line $x+y+z=0$ is not blocked. We finish the paper by illustrating different examples for the case $q\equiv 1 \pmod{4}$.

\begin{ex}\label{ex:degree-m}
Consider the following plane curves of degree $d=\frac{q-1}{2}$,
\begin{itemize}
    \item $x^d+y^d+z^d+(x+y+z)^d+(x-3y+9z)^d+(x+2y+4z)^d=0$ over $\F_{q}$ for $q\in \{13, 17, 29, 37, 53\}$,
     \item $x^d+y^d+z^d+(x+y+z)^d+(x-5y+25z)^d+(x+2y+4z)^d=0$ over $\F_{q}$ for $q\in \{41, 61\}$.
\end{itemize}
Each of these plane curves is smooth and blocking over $\F_q$. \end{ex}

We do not know how to generalize Example~\ref{ex:degree-m} for every $q\equiv 1\pmod{4}$. Interestingly, the curve defined by $$x^{d}+y^{d}+z^{d}+(x+y+z)^{d}+(x-3y+9z)^{d}+(x+2y+4z)^{d}=0$$ is not a blocking curve over $\F_{q}$ for $q\in \{41, 61\}$. Moreover, it is not smooth for $q=61$. It is reasonable to conjecture that for $q\equiv 1\pmod{4}$, there exists a smooth blocking curve of the form,
$$x^{\frac{q-1}{2}}+y^{\frac{q-1}{2}}+z^{\frac{q-1}{2}}+(x+y+z)^{\frac{q-1}{2}}+(x-ay+a^2 z)^{\frac{q-1}{2}}+(x+by+b^2z)^{\frac{q-1}{2}}=0$$
for suitable choices of $a, b\in \F_q$.

\section*{Acknowledgements}
The authors thank Greg Martin and J\'ozsef Solymosi for helpful discussions. During the preparation of this manuscript, the first author was supported by a postdoctoral research fellowship from the University of British Columbia and the NSERC PDF award.
The second author is supported by an NSERC Discovery grant. 

\section*{Data Availability Statement} 
Data sharing is not applicable to this article as no datasets were generated or analysed during the current study.

\bibliographystyle{alpha}
\bibliography{biblio.bib}

\end{document}